\documentclass[a4paper,pdftex,reqno]{amsart}
\usepackage{geometry}
\title{Puncture loops on a non-orientable surface}
\author{Aoi Wakuda}

\makeatletter
\@namedef{subjclassname@2020}{\textup{2020} Mathematics Subject Classification}
\makeatother
\address{GRADUATE SCHOOL OF MATHEMATICAL SCIENCES, UNIVERSITY OF TOKYO, 3-8-1 KOMABA, MEGURO-KU, TOKYO, 153-8914,
JAPAN}
\email{aoichan19991226@g.ecc.u-tokyo.ac.jp}
\usepackage{amsmath} %数式使うならたぶん必要
\usepackage{amssymb} %数式使うならたぶん必要
\usepackage{amsfonts} %数式使うならたぶん必要
\usepackage{mathrsfs} %花文字とか使うのに必要
\usepackage{amsthm} %定理環境に必要
\usepackage{bm} %\bm{a}でaが太字の斜体になる
\usepackage[new]{old-arrows} %たぶん無くても良い
\usepackage[dvipdfmx]{graphicx} %図の挿入に必要
\usepackage{wrapfig} %図をまわりこむように文章を書きたいとき用
\usepackage{color} %文章の色を変えたいとき必要
\usepackage[labelsep=colon]{caption} %図のラベルに関する設定、あんまり気にしなくていい
\usepackage[labelformat=empty,subrefformat=parens]{subcaption} %図を複数並べるときの設定
\usepackage{multicol} %多段組にしたいとき用
\usepackage{units} %なんだっけこれ、忘れた、無くてもいいかも
\usepackage[all]{xy} %可換図式描くのに必要
\usepackage{tikz-cd} %可換図式描くのに必要、他に良いやつがあるかも？
\usepackage{hyperref} %引用をクリックするとそこに飛べるとかいうやつ
\usepackage{color} %色付き文字
\usepackage{undertilde} %　下チルダ
\usepackage{latexsym}
\usepackage{mathtools}
\usepackage{comment}
\usepackage{tikz-cd}
\usepackage{tikz}
\usepackage{float}
\usetikzlibrary{decorations.markings}
\usetikzlibrary{arrows,arrows.meta}
\usepackage{graphicx}

\newtheorem{thm}{Theorem}[section]

\newtheorem{prop}[thm]{Proposition}
\newtheorem{lemma}[thm]{Lemma}
\newtheorem{cor}[thm]{Corollary}

\newtheorem*{mthm*}{Main Theorem}

\theoremstyle{definition}

\newtheorem{ex}[thm]{Example}

\newcommand{\Al}{\alpha}
\newcommand{\Be}{\beta}

\newcommand{\vEp}{\varepsilon}
\newcommand{\tr}{\!\operatorname{Tr}}
\newcommand{\cscd}{\operatorname{csc}}
\newcommand{\cotd}{\operatorname{cot}}

\begin{document}
\begin{abstract}
 On a connected surface \( N \) with negative Euler characteristic, the free homotopy class of a loop obtained by smoothing an intersection of two closed geodesics may wind around a puncture. Chas and Kabiraj showed that this phenomenon does not occur when the surface \( N \) is orientable. In this paper, we prove that it occurs when \( N \) is non-orientable and both geodesics involved in the smoothing are actually one-sided. In particular, we study a loop obtained by traversing a one-sided closed geodesic and the $m$-th power of another one-sided closed geodesic for odd $m$. Then we show that its free homotopy class may wind aroud a puncture at most two values of $m$. Furthermore, if two such \( m\)'s exist, they are consecutive odd integers.
\end{abstract}
\maketitle

\section{Introduction}

The study of closed geodesics on hyperbolic surfaces has been an area of interest in geometry and topology. The structure of free homotopy classes of loops obtained by smoothing intersections of closed geodesics has been explored in various contexts. Chas and Kabiraj \cite{Chas-Kabiraj} showed that for an orientable hyperbolic surface, the smoothing of two closed geodesics never produces a loop homotopic to a puncture. This result naturally raises the question of whether a similar phenomenon holds in the non-orientable case. In this paper, we assume that hyperbolic metric is complete.

For non-orientable surfaces, understanding loops provides insight into their geometric and topological properties. The fundamental groups of such surfaces contain elements corresponding to glide-reflections rather than purely hyperbolic transformations, which affects the structure of closed geodesics and their intersections. Norbury \cite{Norbury2008} generalized McShane’s identity \cite{McShane1991} to non-orientable surfaces. 
In his work, starting from the trace identity given as (11) in \cite[Section 3]{Norbury2008}, 
he derived a length identity that relates two geodesics with a transverse intersection point on a non-orientable surface 
and the geodesic representative of the loop obtained by smoothing it. 
His approach is based on trace identities involving commutators in $PGL(2, \mathbb{R})$. 
In this paper, we take a direct approach of computing hyperbolic lengths without relying on trace identities. It reveals that the loop obtained by smoothing an intersection point of two closed geodesics on a non-orientable surface can sometimes be freely homotopic to a puncture.

In this paper, we also study two closed geodesics that intersect transversely at a point \( P \) on a non-orientable hyperbolic surface. We consider the loop obtained by traversing one of them and traversing the other \( m \) times before smoothing at \( P \). We prove that the integers \( m \) for which this new loop is freely homotopic to a puncture are odd, and such integers occur at most twice. Furthermore, if two such values exist, they must be consecutive odd integers. Our approach is rooted in hyperbolic geometry, where the isometry corresponding to a closed one-sided geodesic is a glide-reflection. By using the classification of isometries in $\operatorname{PSL}_2^{\pm}(\mathbb{R})$ and their translation lengths, we establish explicit conditions on the forward angle at the intersection point that determine whether the resulting loop can be a puncture loop.

%Our results build upon classical work in hyperbolic geometry, particularly the trace formula for the product of hyperbolic isometries \cite{Beardon1983}, as well as recent developments in the study of closed geodesics on non-orientable surfaces \cite{Ohsika2021}. Moreover, they provide new observations on the interaction between geodesic flows and the algebraic structure of non-orientable surface groups.
As an application, in our next paper, we will determine the center of the Goldman Lie algebra with \(\mathbb{Z}/2\mathbb{Z}\) coefficients by using Lemma \ref{general_intersection}.

\noindent
\textbf{Organization of the paper.} 
In section 2, we provide a brief review of hyperbolic geometry. Additionally, we examine some relation between glide-reflections and orientation-preserving hyperbolic elements. In section 3, we investigate the phenomenon where a loop, obtained by smoothing an intersection of two one-sided geodesics, may wind around a puncture. 

\noindent
\textbf{Acknowledgment.} The author would like to thank my supervisor, Nariya Kawazumi, for many discussions and helpful advice. The author is also grateful to Ken'ichi Ohsika for providing valuable comments and to Toyo Taniguchi for engaging in insightful discussions. This work was supported by the WINGS-FMSP program at the Graduate School of Mathematical Sciences, the University of Tokyo.
\section{isometries on hyperbolic plane}
Let $\mathbb{H}$ be the upper half-plane model in hyperbolic geometry. The group of isometries of $\mathbb{H}$ is given by
\[
\operatorname{PSL}^{\pm}_2(\mathbb{R}) = \left\{ \pm \begin{pmatrix} a & b \\ c & d \end{pmatrix} \middle| a, b, c, d \in \mathbb{R}, ad - bc = \pm 1 \right\}.
\] 
Each isometry \( g \in \operatorname{PSL}^{\pm}_2(\mathbb{R}) \) is classified according to the absolute value of its trace (for example, see \cite{Wilkie1966}):

- Case 1: \( \det g = 1 \) (Orientation-preserving isometries)
  \begin{itemize}
      \item \textbf{Elliptic} if $|\operatorname{Tr} g| < 2$. In this case, $g$ has a unique fixed point in $\mathbb{H}$.
      \item \textbf{Parabolic} if $|\operatorname{Tr} g| = 2$. In this case, $g$ has a unique fixed point on the real axis.
      \item \textbf{Hyperbolic} if $|\operatorname{Tr} g| > 2$. In this case, $g$ has exactly two fixed points on the real axis.
  \end{itemize}

- Case 2: \( \det g = -1 \) (Orientation-reversing isometries)
  \begin{itemize}
      \item \textbf{Reflection} if $\operatorname{Tr} g = 0$. In this case, $g$ is an involution with a geodesic of fixed points.
      \item \textbf{Glide-reflection} if $\operatorname{Tr} g \neq 0$. In this case, $g$ has two fixed points on the real axis.
  \end{itemize}

For each \( g \in \operatorname{PSL}^{\pm}_2(\mathbb{R}) \), the translation length \( t_g \) is defined as:
\[
t_g = \inf_{z \in \mathbb{H}} d(z, gz).
\]
If $t_g$ is positive, we call $g$ a \textit{positive translation isometry}.  
For a positive translation isometry \( g \), the absolute value of the trace of \( g \) satisfies the following:
\begin{align} \label{Trcoshsinh}
|\!\operatorname{Tr} g| =
\begin{cases}
2 \cosh \dfrac{t_g}{2}, & \text{if } g \text{ is hyperbolic}, \\
2 \sinh \dfrac{t_g}{2}, & \text{if } g \text{ is a glide-reflection}.
\end{cases}
\end{align}

For a positive translation isometry \( g  \in \operatorname{PSL}^{\pm}_2(\mathbb{R}) \), we define \( A_g \) as the axis of \( g \), i.e., the geodesic joining the two fixed points of \( g \), and let \( \rho_{A_g} \) be the reflection with respect to \( A_g \).

We state the following result from \cite[Theorem 7.38.6]{Beardon1983} because it will be used later.
\begin{thm} \label{thm two-sided two-sided}
\cite[Theorem 7.38.6]{Beardon1983}
Let $g$ and $h$ be hyperbolic transformations of the hyperbolic plane and suppose that $A_{g}$ and $A_{h}$ intersect at a point $P$. Denote by $\theta_P$ the angle at P between forward direction of $A_{g}$ and $A_{h}$. Then the composition $g\circ h$ is hyperbolic and 
\begin{align} \label{Trcoshcoshsinhsinhcos}
\frac{1}{2}|\!\operatorname{Tr} gh| = \cosh\left(\frac{t_{g}}{2}\right)\cosh\left(\frac{t_{h}}{2}\right) + \sinh\left(\frac{t_{g}}{2}\right)\sinh\left(\frac{t_{h}}{2}\right)\cos(\theta_P).
\end{align}
\end{thm}
Now, we generalize the above theorem to positive translation isometries.
The proof of \cite[Theorem 7.38.6]{Beardon1983} uses the law of cosines for the triangle formed by the three axes \(A_g\), \(A_h\), and \(A_{gh}\) to compute the absolute value of the trace. In this paper we compute the absolute value of the trace by matrix computations as follows.
%While the original proof is geometric, we provide a proof based on matrix calculations below.

\begin{thm} \label{Trgh}
Let $g$ and $h$ be positive translation isometries and suppose that $A_{g}$ and $A_{h}$ intersect at a point $P$.  Denote by $\theta_P$ the angle at P between forward direction of $A_{g}$ and $A_{h}$. Then, the following hold.

\noindent \textbf{Case 1:} If \(g\) is a glide-reflection and \(h\) is a hyperbolic element, 
\begin{align} \label{Trsinhcoshcoshsinhcos}
    \frac{1}{2}|\!\operatorname{Tr} gh| = \left|\sinh\left(\frac{t_g}{2}\right)\cosh\left(\frac{t_h}{2}\right) + \cosh\left(\frac{t_g}{2}\right)\sinh\left(\frac{t_h}{2}\right)\cos\theta_P\right|.
\end{align} \label{Trsinhsinhcoshcoshcos}
\textbf{Case 2:} If both \(g\) and \(h\) are glide-reflections, 
\begin{align}
    \frac{1}{2}|\!\operatorname{Tr} gh| = \left|\sinh\left(\frac{t_g}{2}\right)\sinh\left(\frac{t_h}{2}\right) + \cosh\left(\frac{t_g}{2}\right)\cosh\left(\frac{t_h}{2}\right)\cos\theta_P\right|.
\end{align}
\end{thm}

\begin{proof}
Without loss of generality, we may assume the axis \( A_g \) equals the imaginary axis with the intersection point \( P \in A_g \cap A_h \) located at \( i \in \mathbb{H}\) . See Figure \ref{fig_trace_of_gh}.

\begin{figure}[H]
\centering\includegraphics[width=0.6\linewidth]{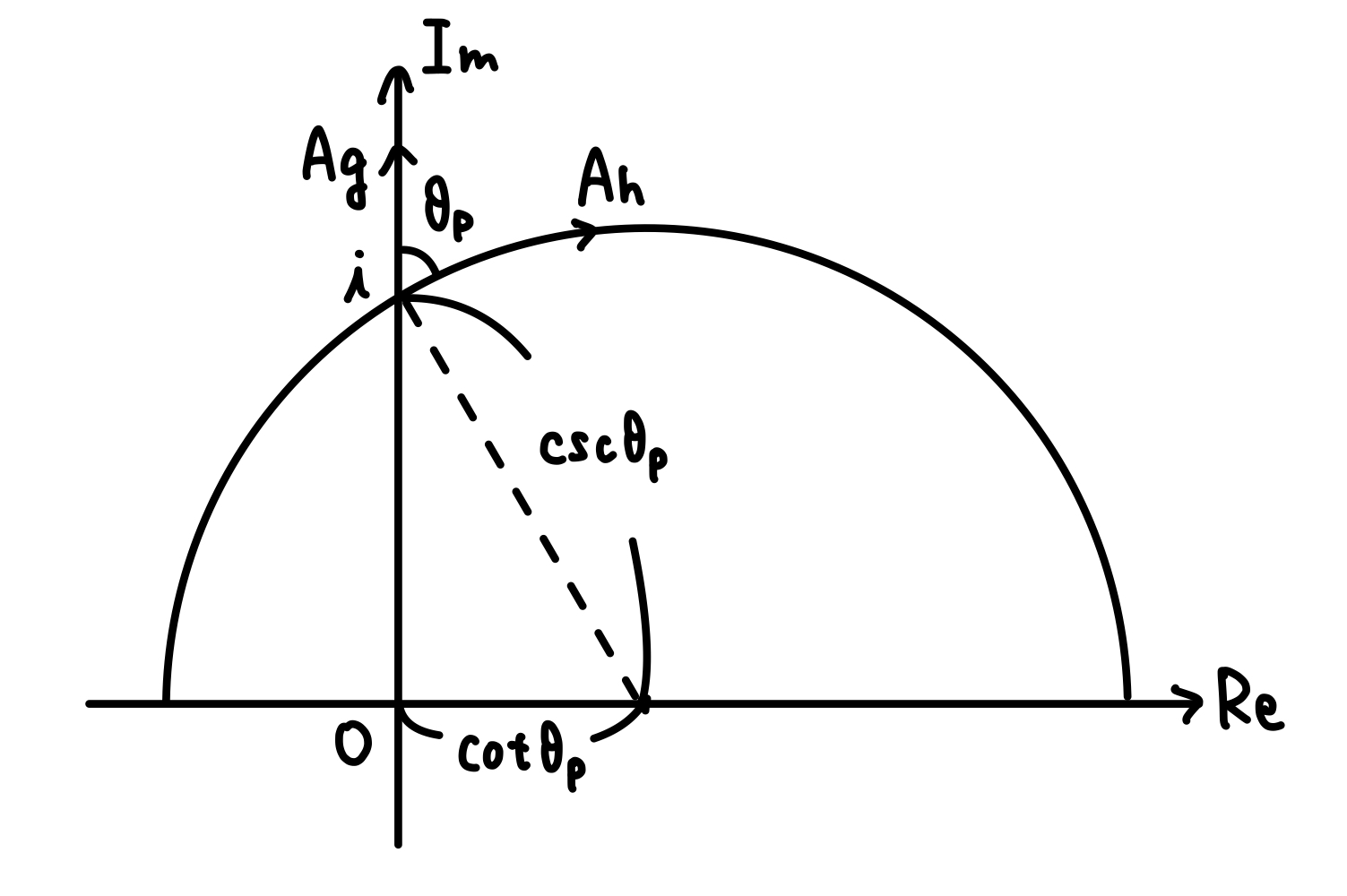}
\caption{Two intersecting axes \( A_g \) and \( A_h \) in the upper half-plane \( \mathbb{H} \).}
\label{fig_trace_of_gh}
\end{figure}

First, we consider the case (1), where $g$ is a glide-reflection and $h$ is a hyperbolic element:
\[
g = 
\begin{pmatrix}
e^{t_g/2} & 0 \\
0 & -e^{-t_g/2}
\end{pmatrix}, \quad
h = 
\begin{pmatrix}
p & q \\
r & s
\end{pmatrix}, \quad
\text{where } ps - qr = 1 \text{ and } p+s > 0.
\]
From the definition of the trace,
\[
|\tr h| = p + s = 2 \cosh \frac{t_h}{2}.
\]
Let $z$ be the fixed point of $h$. By Figure \ref{fig_trace_of_gh},
\[
z = \cotd \theta_p \pm \cscd \theta_p,
\]
Moreover, from the definition of the fixed point,
\[
\frac{pz+q}{rz+s} = z.
\]
By the relation between the roots and the coefficients, we have
\[
\begin{cases}
p - s = 2r \cotd \theta_p, \\
q = r.
\end{cases} 
\]
Hence,
\[
p = \cosh \frac{t_h}{2} + \sinh \frac{t_h}{2} \cos \theta_p, \quad 
q = r = \sinh \frac{t_h}{2} \sin \theta_p, \quad
s = \cosh \frac{t_h}{2} - \sinh \frac{t_h}{2} \cos \theta_p.
\]
Therefore, we get
\[
\frac{1}{2}|\!\operatorname{Tr} gh| = \left|\sinh\left(\frac{t_g}{2}\right)\cosh\left(\frac{t_h}{2}\right) + \cosh\left(\frac{t_g}{2}\right)\sinh\left(\frac{t_h}{2}\right)\cos\theta_P\right|.
\]
Next, we consider the case (2), where both $g$ and $h$ are glide-reflections. The reflection by \( A_h \) is given by
\[
\begin{pmatrix}
-\cos\theta_p & -\sin\theta_p \\
-\sin\theta_p & \cos\theta_p
\end{pmatrix},
\]
so we can write
\[
\begin{aligned}
h &=
\begin{pmatrix}
-\cos\theta_p & -\sin\theta_p \\
-\sin\theta_p & \cos\theta_p
\end{pmatrix}
\begin{pmatrix}
\cosh\dfrac{t_h}{2} + \sinh\dfrac{t_h}{2}\cos\theta_p & \sinh\dfrac{t_h}{2}\sin\theta_p \\
\sinh\dfrac{t_h}{2}\sin\theta_p & \cosh\dfrac{t_h}{2} - \sinh\dfrac{t_h}{2}\cos\theta_p
\end{pmatrix} \\
&=
\begin{pmatrix}
-\cosh\dfrac{t_h}{2}\cos\theta_p - \sinh\dfrac{t_h}{2} & -\cosh\dfrac{t_h}{2}\sin\theta_p \\
-\cosh\dfrac{t_h}{2}\sin\theta_p & \cosh\dfrac{t_h}{2}\cos\theta_p - \sinh\dfrac{t_h}{2}
\end{pmatrix}.
\end{aligned}
\]
Thus, we have
\[
\frac{1}{2}|\!\operatorname{Tr} gh| = \left|\sinh\left(\frac{t_g}{2}\right)\sinh\left(\frac{t_h}{2}\right) + \cosh\left(\frac{t_g}{2}\right)\cosh\left(\frac{t_h}{2}\right)\cos\theta_P\right|.
    \]
\end{proof}

On the other hand, when considering the relative positions of the axes, a geometric approach, similar to the proof of \cite[Theorem 7.38.6]{Beardon1983}, provides a clearer understanding. Thus, we adopt this approach in the proof.

\begin{thm} \label{thm_hyp_gli}
Let $g$ be a hyperbolic element and $h$ be a glide-reflection. Suppose that $A_{g}$ and $A_{h}$ intersect at a point $P$. Denote by $\theta_P$ the angle at P between forward direction of $A_{g}$ and $A_{h}$. Then the axis \( A_{gh} \) intersects transversely with the axis \( A_g \) at the point \( Q \) located at a distance of \( t_g/2 \) from \( P \) in the direction of \( g \).
\end{thm}

To prove Theorem \ref{thm_hyp_gli}, we recall some basic propositions in hyperbolic geometry. It is easy to see that these propositions follow from the properties of hyperbolic isometries.

\begin{prop} \label{rot. and rot.}
In the hyperbolic plane $\mathbb{H}$, an isometry $g\in G$ is hyperbolic if and only if $g=\vEp_2 \vEp_1$ where $\vEp_1,\vEp_2$ are rotations of order two about some distinct points $v_1,v_2$. Then, the axis $A_g$ passes through $v_1,v_2$, its direction is from $v_1$ to $v_2$, and $2d(v_1,v_2)=t_g$
\end{prop}

\begin{prop} \label{ref. and ref.}
In the hyperbolic plane $\mathbb{H}$, an isometry $g\in G$ is hyperbolic if and only if $g=\rho_2 \rho_1$where $\rho_1,\rho_2$ are reflections about some disjoint geodesics $L_1,L_2$. Then, the axis $A_g$ is orthogonal to $L_1,L_2$, its direction is from $L_1$ to $L_2$, and $2d(L_1,L_2)=t_g$
\end{prop}
Let \(L\) be a geodesic in \(\mathbb{H}\). For any point \(P \in L\), the rotation \(\vEp_P\) of order two about \(P\) commutes with the reflection \(\rho_L\) about \(L\). In other words, we have
\begin{align} \label{rot_and_ref}
\vEp_P \rho_L = \rho_L \vEp_P.
\end{align}
\begin{comment}
\begin{thm} \label{thm two-sided two-sided}
\cite[Theorem 7.38.6]{Beardon1983}
Let $g$ and $h$ be hyperbolic elements of the hyperbolic plane and suppose that $A_{g}$ and $A_{h}$ intersect at a point $P$. Denote by $\theta_P$ the angle at P between forward direction of $A_{g}$ and $A_{h}$. Then the composition $g\circ h$ is hyperbolic and 
\begin{align}
\cosh\left(\frac{t_{g\circ h}}{2}\right) = \cosh\left(\frac{t_{g}}{2}\right)\cosh\left(\frac{t_{h}}{2}\right) + \sinh\left(\frac{t_{g}}{2}\right)\sinh\left(\frac{t_{h}}{2}\right)\cos(\theta_P).
\end{align}
\end{thm}
We call an isometric transformation \(f \in Isom^-(\mathbb{H})\) a {\it glide-reflection} if there exists a hyperbolic element \(g \in Isom^+(\mathbb{H})\) such that \(f = \rho_{A_g} \circ g\). As in the case of ordinary hyperbolic elements, we define the axis of a glide-reflection \(f = \rho_{A_g} \circ g \in Isom^-(\mathbb{H})\) to be \(A_g\). Also, the translation length of \(f\) is equal to that of \(g\).
Next, let us observe how the result changes when the usual hyperbolic element in the above theorem is replaced with a glide-reflection. First, let us consider the case where one is a glide-reflection, and the other is a usual hyperbolic element.
\end{comment}
\begin{proof}[Proof of Theorem \ref{thm_hyp_gli}]

Let \( v_1 := P\) and \(v_2 := Q \). Define \( v_3 \) as the point located at a distance of \( \dfrac{t_h}{2} \) along the axis \( A_h \) in its negative direction from \( v_1 \). Next, let \( L_1 \) be the line passing through \( v_3 \) and perpendicular to \( A_h \), and \( L_2 \) the line passing through \( v_2 \) and perpendicular to \( L_1 \). Let \( v_4 \) be the intersection of \( L_1 \) and \( L_2 \) (see Figure \ref{fig.twosided_onesided}). 
\begin{figure}[H]
\centering\includegraphics[width=0.95\linewidth]{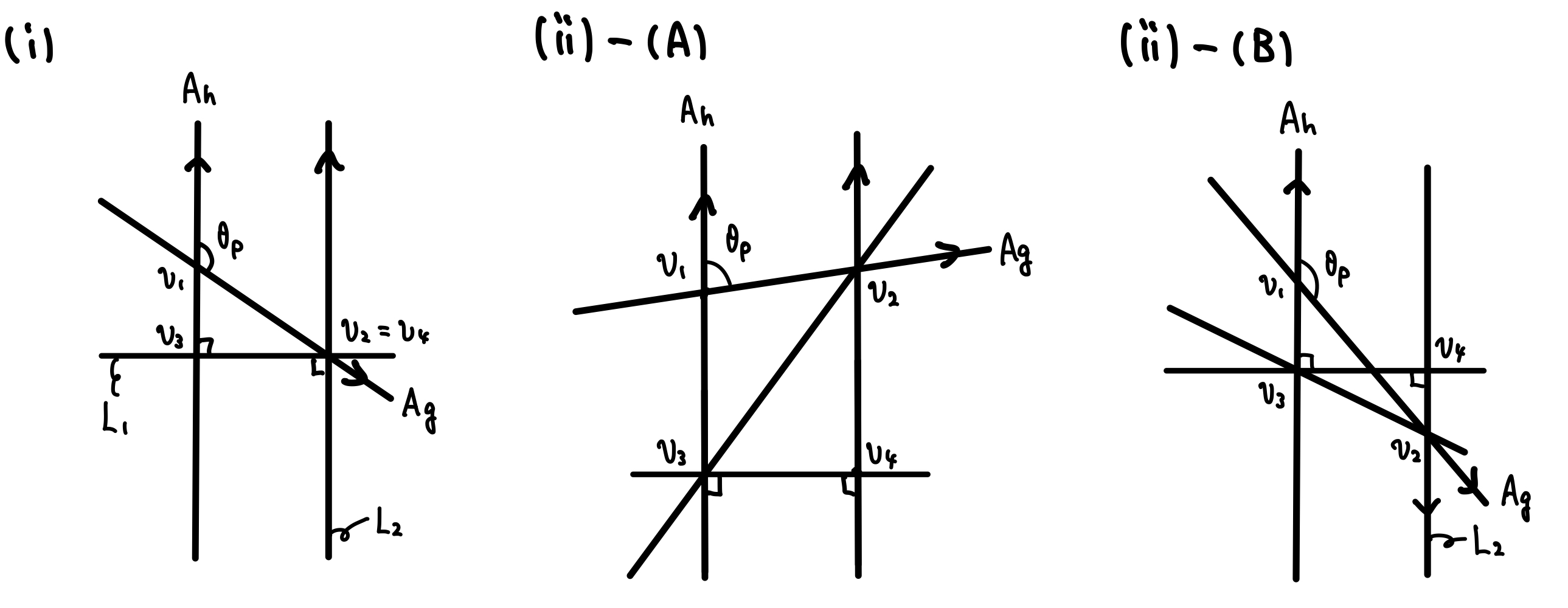}
\caption{\label{fig.twosided_onesided}The value of $\cosh\left(\dfrac{t_g}{2}\right)\sinh\left(\dfrac{t_h}{2}\right) + \sinh\left(\dfrac{t_g}{2}\right)\cosh\left(\dfrac{t_h}{2}\right)\cos\theta_P$ is zero on the left, positive in the middle, and negative on the right.}
%\caption{\label{fig.cosh}Theorem \ref{thm cosh}}
\end{figure}
Then, we have $a = \varepsilon_{v_2} \varepsilon_{v_1}$ and $b = \rho_{v_1 v_3} \varepsilon_{v_1} \varepsilon_{v_3}$. Hence, we have
\[
\begin{aligned}
a \circ b &= \varepsilon_{v_2} \varepsilon_{v_1} \rho_{v_1 v_3} \varepsilon_{v_1} \varepsilon_{v_3} \\
         &= \varepsilon_{v_2} \varepsilon_{v_1} \varepsilon_{v_1} \rho_{v_1 v_3} \varepsilon_{v_3} \quad (\text{by (\ref{rot_and_ref}})) \\
         &= \varepsilon_{v_2} \rho_{v_1 v_3} \varepsilon_{v_3} \\
         &\overset{(*)}{=} \rho_{L_2} \varepsilon_{v_2} \varepsilon_{v_4}.
\end{aligned}
\]
Here, \( \rho_{L_2} \) is a reflection about the line $L_2$. The identity \((*)\) is equivalent to
\begin{align}\label{34132}
\varepsilon_{v_3} \varepsilon_{v_4} = \rho_{v_1 v_3} \rho_{L_2}.
\end{align}

By Proposition \ref{rot. and rot.}, the left-hand side is a hyperbolic element whose axis is \( v_4 v_3 \) with direction from \( v_4 \) to \( v_3 \), and whose translation length is \( 2 d(v_4, v_3) \). Moreover, Proposition \ref{ref. and ref.} shows that the right-hand side also defines a hyperbolic element with the same axis, direction, and translation length. Since a hyperbolic element is uniquely determined by its axis, direction, and translation length, we obtain (\ref{34132})

\end{proof}

By using the same approach as we used in the proof of Theorem \ref{thm_hyp_gli}, prove the following theorem as well.

\begin{thm} \label{thm one-sided two-sided}
Let $g$ be a glide-reflection of the hyperbolic plane and $h$ a hyperbolic element of the hyperbolic plane. Suppose that $A_{g}$ and $A_{h}$ intersect at a point $P$. Denote by $\theta_P$ the angle at P between forward direction of $A_{g}$ and $A_{h}$. Then the axis \( A_{gh} \) intersects transversely with the axis \( A_h \) at the point \( Q \) located at a distance of \( t_h/2 \) from \( P \) in the opposite direction of \( h \). 
\begin{comment}
\begin{enumerate}
  \item[(i)] a reflection if 
  \[
  \sinh\left(\dfrac{t_{g}}{2}\right)\cosh\left(\dfrac{t_{h}}{2}\right) + \cosh\left(\dfrac{t_{g}}{2}\right)\sinh\left(\dfrac{t_{h}}{2}\right)\cos(\theta_P) = 0,
  \]
  \item[(ii)] a glide-reflection if 
  \[
  \sinh\left(\dfrac{t_{g}}{2}\right)\cosh\left(\dfrac{t_{h}}{2}\right) + \cosh\left(\dfrac{t_{g}}{2}\right)\sinh\left(\dfrac{t_{h}}{2}\right)\cos(\theta_P) \neq 0.
  \] 
\end{enumerate}

In particular, if the composition $g\circ h$ is a glide-reflection, then

\begin{align}
\sinh\left(\frac{t_{g\circ h}}{2}\right) = |\sinh\left(\frac{t_{g}}{2}\right)\cosh\left(\frac{t_{h}}{2}\right) + \cosh\left(\frac{t_{g}}{2}\right)\sinh\left(\frac{t_{h}}{2}\right)\cos(\theta_P)|.
\end{align}
\end{comment}
\end{thm}

\section{Loops on a non-orientable surface and hyperbolic geometry}

In this section, we study some relation between loops on a non-orientable surface and glide-reflections.

Let $N$ be a non-orientable surface (possibly with punctures). We assume the Euler characteristic of $N$ is negative so that $N$ admits a complete hyperbolic metric. We call a closed curve a {\it puncture loop} if it is freely homotopic to a loop that winds around a puncture one or more times. A closed curve that is neither a puncture loop nor a trivial loop is called {\it essential}. For an essential free homotopy class $\Al$ , $\Al(X)$ denotes a geodesic representative of $\Al$ with respect to a complete hyperbolic metric $X$.

\subsection{Forward angles and lengths.}

For each transverse intersection point $P$ of \( \Al(X) \) and \( \Be(X) \), {\it the X forward angle of $\Al$ and $\Be$ at $P$}, denoted by $\theta_{P}(X)$, is defined as the angle between the directions of \( \Al(X) \) and \( \Be(X) \). See Figure \ref{fig.forward_angle_by_metric}.
\begin{figure}[H]
\label{Angle}
\centering\includegraphics[width=0.4\linewidth]{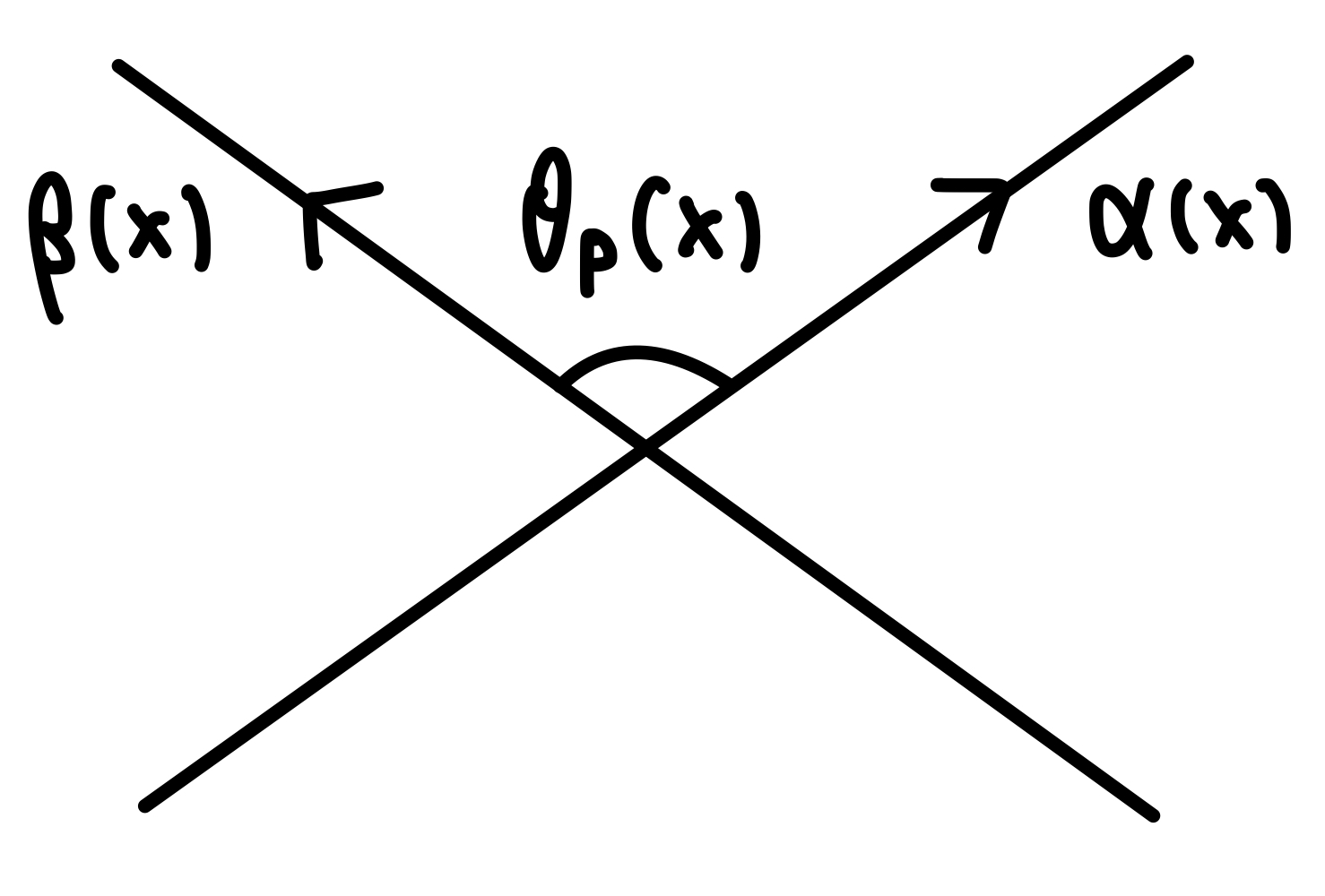}
\caption{\label{fig.forward_angle_by_metric}The forward angle $\theta_{P}(X)$}
\end{figure}

For each complete hyperbolic metric $X$ and each $\Al$ $\in$ $\hat{\pi}$, denote by $l_{\Al}(X)$ the length of $\Al(X)$. To simplify notation, we write $l_{\Al}$ instead of $l_{\Al}(X)$ if there is no confusion and similarly for $\theta_{P}(X)$.

\subsection{Smoothing an intersection of two closed geodesics}

Theorem \ref{thm two-sided two-sided} and Theorem \ref{Trgh} imply the following.

\begin{lemma} \label{general_intersection}
Let $X$ be a complete hyperbolic metric on $N$ and $\Al, \Be \in \hat{\pi}$. Suppose that $P$ is a transverse intersection point of $\Al(X)$ and $\Be(X)$. Then, the following holds:

\textbf{Case 1:} If both $\Al$ and $\Be$ are two-sided, then the free homotopy class $|\Al_P\Be_P|$ is essential. Moreover, $\cosh\left(\dfrac{l_{|\Al_P\Be_P|}}{2}\right)$ satisfies the equation coming from (\ref{Trcoshsinh}) and (\ref{Trcoshcoshsinhsinhcos}).

\textbf{Case 2:} If $\Al$ is one-sided and $\Be$ is two-sided, then the free homotopy class $|\Al_P\Be_P|$ is essential. Moreover, $\sinh\left(\dfrac{l_{|\Al_P\Be_P|}}{2}\right)$ satisfies the equation coming from (\ref{Trcoshsinh}) and (\ref{Trsinhcoshcoshsinhcos}).

\textbf{Case 3:} Let both $\Al$ and $\Be$ be one-sided.
\noindent If the quantity
\[
\left|\sinh\left(\frac{l_{\Al}}{2}\right)\sinh\left(\frac{l_{\Be}}{2}\right) + \cosh\left(\frac{l_{\Al}}{2}\right)\cosh\left(\frac{l_{\Be}}{2}\right)\cos\theta_{P}\right|,
\]
equals $1$, then the free homotopy class $|\Al_P\Be_P|$ is a puncture loop.

On the other hand, if the above quantity is strictly greater than $1$, then the free homotopy class $|\Al_P\Be_P|$ is essential. Moreover, $\cosh\left(\dfrac{l_{|\Al_P\Be_P|}}{2}\right)$ equals the above quantity.
\end{lemma}
For the orientable case, see \cite[Lemma 2.7]{Chas-Kabiraj}.

\begin{thm} \label{mainthm1}
Let $X$ be a complete hyperbolic metric on $N$ and $\Al,\Be \in \hat{\pi}$. Suppose \( \Al(X) \) and \( \Be(X) \) intersect at a point $P$ transversely, and $|\Al^m_P\Be_P|$ is a puncture loop, for some non-zero integer $m$. Then such integers are odd, and the number of nonzero integers $m$ is at most two. Moreover, if the number is two, the odd numbers must be consecutive.
\end{thm}

\begin{proof}
By Lemma \ref{general_intersection}, it follows that both \( \alpha^m \) and \( \beta \) must be one-sided and an odd integer $m$ satisfies

\begin{align}
\left|\sinh\left(m\frac{l_{\Al}}{2}\right)\sinh\left(\frac{l_{\Be}}{2}\right) + \cosh\left(m\frac{l_{\Al}}{2}\right)\cosh\left(\frac{l_{\Be}}{2}\right)\cos\theta_P\right|=1.
\end{align}
In particular, when \( m \) is negative, the forward angle should must be changed. However, even in this case, the above equation holds.

Let
\[
f(t) = \left|\sinh (t) \sinh\left(\frac{l_{\Be}}{2}\right) + \cosh (t) \cosh \left(\frac{l_{\Be}}{2}\right) \cos \theta_P\right|.
\]
The number of nonzero integers \( m \) that satisfy the given conditions, is at most the number of times \( f(t) \) crosses the value 1. If we denote \( r = \sinh\left(\dfrac{l_{\Be}}{2}\right) \) and \( s = \cosh\left(\dfrac{l_{\Be}}{2}\right) \cos \theta_P,\)
we have 
\[
f(t) = |r \sinh t + s \cosh t |,
\]
The graph \( y = f(t) \) is given as in Figure \ref{fig_five_functions_of_f} depending on the values of \( r \) and \( s \). Therefore, by Figure \ref{fig_five_functions_of_f}, the number of nonzero integers $m$ satisfying the condition is at most two.

\begin{figure}[H]
\label{fig_five_functions_of_f}
\centering\includegraphics[width=1\linewidth]{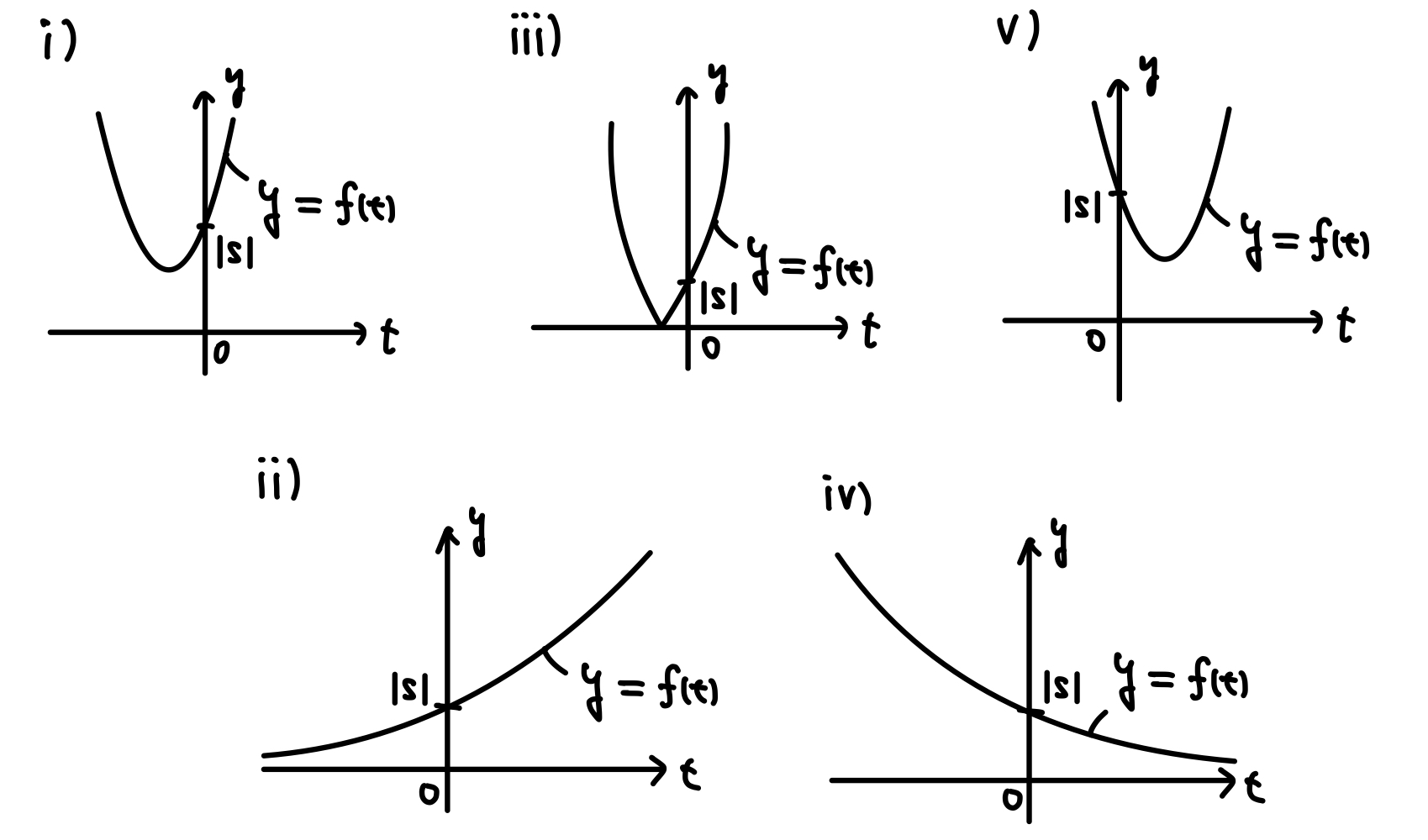}
\caption{\label{fig_five_functions_of_f}Five cases of the function \( f(t) \), determined by the values of \( r \) and \( s \): (i) \( s > r \), (ii) \( s = r \), (iii) \( -r < s < r \), (iv) \( s = -r \), (v) \( s < -r \).}
\end{figure}

Finally, let \(m_1\) and \(m_2\) be such odd numbers. Then, the graph of \(y = f(t)\) fits only one of the three patterns (i), (iii), and (v), and satisfies
\[
f\left(m_1\frac{l_{\Al}}{2}\right) = f\left(m_2\frac{l_{\Al}}{2}\right) = 1.
\]
Here, assuming that \(m_1\) and \(m_2\) are not consecutive, there exists an odd number \(m_3\) between \(m_1\) and \(m_2\). In particular, even in the case of the above three graphs,
\[
f\left(m_3\frac{l_{\Al}}{2}\right) < 1
\]
holds. This leads to a contradiction.
\end{proof}

\begin{ex}
The left example in Figure \ref{fig_puncture_loop_example} shows the case where only $|\Al_P\Be_P|$ is a puncture loop.  
The middle example in Figure \ref{fig_puncture_loop_example} shows the case where both $|\Al_P\Be_P|$ and $|\Al^{-1}_{P}\Be_P|$ are puncture loops.  
In particular, the forward angle at $P$ is $\dfrac{\pi}{2}$.  
The right example in Figure \ref{fig_puncture_loop_example} shows the case where both $|\Al_P\Be_P|$ and $|\Al^{3}_{P}\Be_P|$ are puncture loops.
In particular, the forward angle at $P$ is greater than $\dfrac{\pi}{2}$.  
\begin{figure}[H]
\label{Puncture_loop_example}
\centering
\includegraphics[width=1\linewidth]{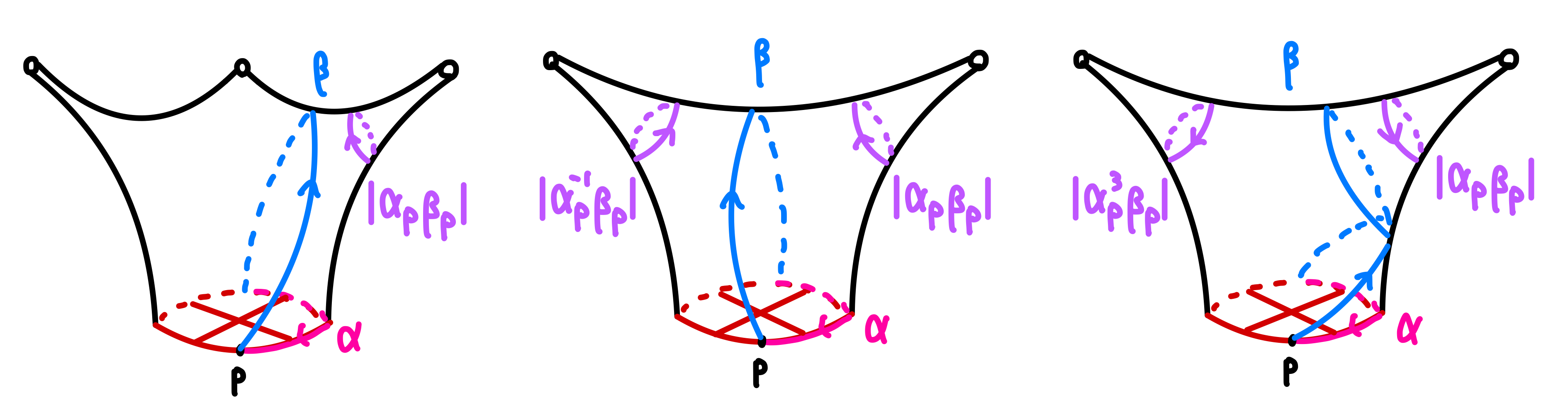}
\caption{\label{fig_puncture_loop_example}Puncture loops obtained by smoothing an intersection of two closed geodesics.}
\end{figure}

Theorem \ref{mainthm1} states that the nonzero odd integers \( m \) for which \( |\alpha^m_P \beta_P| \) is a puncture loop are at most two and must be consecutive. This naturally raises the question of how the forward angle \( \theta_P \) varies depending on the winding number \( m \). In particular, the following theorem describes how \( \theta_P \) behaves when two consecutive odd integers satisfy the puncture loop condition.

\begin{thm}
Let $X$ be a complete hyperbolic metric on $N$ and $\Al,\Be \in \hat{\pi}$. Denote by $P$ a transverse intersection point of \( \Al(X) \) and \( \Be(X) \), and let  \( m \in 2\mathbb{Z} \). If the classes $\left| \alpha^{m+1}_p \beta_p \right|$ and $\left| \alpha^{m-1}_p \beta_p \right|$ are puncture loops, then the followings hold:
\[
\begin{cases}
\theta_P > \dfrac{\pi}{2} & (m > 0), \\[6pt]
\theta_P = \dfrac{\pi}{2} & (m = 0), \\[6pt]
\theta_P < \dfrac{\pi}{2} & (m < 0).
\end{cases}
\]
\end{thm}

\begin{proof}
First, if \( m > 0 \), then we define a function \( g(t) \) by  
\[
g(t) = \dfrac{1 - \sinh t \sinh \left( \dfrac{l_{\beta}}{2} \right)}{\cosh t \cosh \left( \dfrac{l_{\beta}}{2} \right)}.
\]
Then,
\[
g'(t) = \dfrac{-\sinh \left( \dfrac{l_{\beta}}{2} \right) - \sinh t}{\cosh^2 t \cosh \left( \dfrac{l_{\beta}}{2} \right)}.
\]
Since \( g(t) \) is a strictly decreasing monotonic function on $\mathbb{R}_{>0}$, we obtain  
\[
\theta_P > \dfrac{\pi}{2}.
\]

Next if \( m = 0 \), then we have  
\begin{align*}
\left| \sinh \left( \dfrac{l_{\alpha}}{2} \right) \sinh \left( \dfrac{l_{\beta}}{2} \right) + \cosh \left( \dfrac{l_{\alpha}}{2} \right) \cosh \left( \dfrac{l_{\beta}}{2} \right) \cos \theta_P \right| = 1,\\
\left| -\sinh \left( \dfrac{l_{\alpha}}{2} \right) \sinh \left( \dfrac{l_{\beta}}{2} \right) + \cosh \left( \dfrac{l_{\alpha}}{2} \right) \cosh \left( \dfrac{l_{\beta}}{2} \right) \cos \theta_P \right| = 1.
\end{align*}
Since these two equations hold, we get
\begin{align*}
\sinh \left( \dfrac{l_{\alpha}}{2} \right) \sinh \left( \dfrac{l_{\beta}}{2} \right) + \cosh \left( \dfrac{l_{\alpha}}{2} \right) \cosh \left( \dfrac{l_{\beta}}{2} \right) \cos \theta_P = 1,\\
-\sinh \left( \dfrac{l_{\alpha}}{2} \right) \sinh \left( \dfrac{l_{\beta}}{2} \right) + \cosh \left( \dfrac{l_{\alpha}}{2} \right) \cosh \left( \dfrac{l_{\beta}}{2} \right) \cos \theta_P = -1.
\end{align*}
Therefore, we obtain
\[
\theta_P = \dfrac{\pi}{2}.
\]

Finally If \( m < 0 \), then we define a function \( h(t) \) by  
\[
h(t) = \dfrac{-1 - \sinh t \sinh \left( \dfrac{l_{\beta}}{2} \right)}{\cosh t \cosh \left( \dfrac{l_{\beta}}{2} \right)}.
\]
Then,
\[
h'(t) = \dfrac{-\sinh \left( \dfrac{l_{\beta}}{2} \right) + \sinh t}{\cosh^2 t \cosh \left( \dfrac{l_{\beta}}{2} \right)}.
\]
Since \( h(t) \) is a strictly decreasing monotonic function on $\mathbb{R}_{<0}$, we obtain  
\[
\theta_P < \dfrac{\pi}{2}.
\]
\end{proof}
\end{ex}
\subsection{Self-intersection number of a loop on a non-orientable surface} Finally, we consider the self-intersections of a loop on a non-orientable surface. When a closed geodesic intersects another one-sided closed geodesic transversely, we can evaluate the number of self-intersection points. The following proposition provides an explicit lower bound for this number.
\begin{prop}
Let $X$ be a complete hyperbolic metric on $N$ and $\Al,\Be \in \hat{\pi}$ with \( \beta \) one-sided. Suppose \( \Al(X) \) and \( \Be(X) \) intersect at a point $P$ transversely, and $\left| \cos \theta_{P} \right| > \tanh \left( \dfrac{(2m-1) l_{\beta}}{2} \right)$. Then $\Al$ has at least $m$ self-intersection points.
\end{prop}
\begin{proof}
Let \( \tilde{P}_{0} \) be a lift of \( P \) to the universal cover.  
Let \( \tilde{\beta} \) be the geodesic on the universal cover which covers \( \beta(X) \) and passes through \( \tilde{P}_{0} \). For any integer or half-integer $i \geq 0$, we denote by \( \tilde{P}_{i} \) the point obtained by moving from $P_0$ in the direction opposite to $\beta$ by $i \cdot l_{\beta}$ and by $\tilde{\alpha}_i$ the geodesic which covers $\alpha(X)$ and passes through $\tilde{P}_i$. Furthermore, for any half-integer $i$, let $L_{i}$ be the geodesic passing through $\tilde{P}_i$ and orthogonal to \( \tilde{\beta} \). See Figure \ref{fig_self-intersection_lift_1}.
\begin{figure}[H]
\label{Self-intersection}
\centering\includegraphics[width=0.9\linewidth]{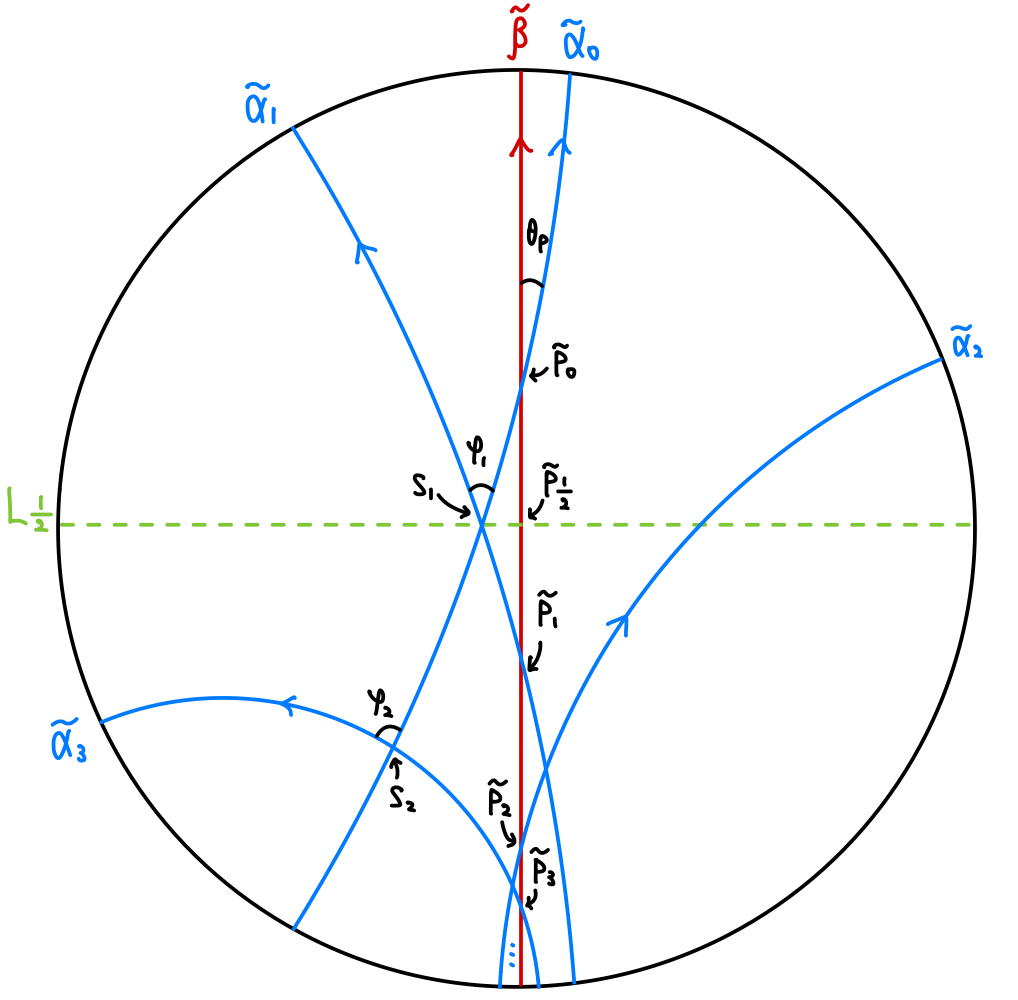}
\caption{\label{fig_self-intersection_lift_1}Lifts of $P$, $\alpha(X)$, and $\beta(X)$ in the Poincaré disk model.
}
\end{figure}

(i) If \( \cos \theta_{P} > \tanh \left( \dfrac{(2m-1) l_{\beta}}{2} \right) \), then \( \tilde{\alpha}_0 \) intersects the geodesics \( L_{\frac{2i-1}{2}}\) (\( i = 1, \dots, m \)) transversely.  
Let \( S_i \) (\( i = 1, \dots, m \)) be these intersection points.  
Since \( \beta \) is one-sided, the reflection with respect to \( L_{\frac{2i-1}{2}}\) maps $\tilde{\alpha}_{2i-1}$ to $\tilde{\alpha}_0$. Hence, $\tilde{\alpha}_{2i-1}$ and $\tilde{\alpha}_0$ intersect transversely at $S_i$. By the Gauss-Bonnet theorem, the forward angle $\varphi_i$ at $S_i$ satisfies
\[
\varphi_1 < \varphi_2 < \dots < \varphi_m.
\]
This implies that \( \alpha \) has at least \( m \) self-intersections.

(ii) If \( \cos \theta_{P} < -\tanh \left( \dfrac{(2m-1) l_{\beta}}{2} \right) \), the same argument as in (i) applies.
\end{proof}

By substituting \( m = 1 \) into the above proposition, we obtain the following corollary.

\begin{cor}
Let $X$ be a complete hyperbolic metric on $N$ and $\Al,\Be \in \hat{\pi}$ with \( \beta \) one-sided. Suppose \( \Al(X) \) and \( \Be(X) \) intersect at a point $P$ transversely, and $\left| \cos \theta_{P} \right| > \tanh \left( \dfrac{l_{\beta}}{2} \right)$. Then $\Al$ is not simple.
\end{cor}

\bibliography{PLONS}
\bibliographystyle{plain}
%参考文献リスト
%sample.bibというファイルから\cite{〇〇}で該当する文献を引用できる

\begin{comment}
\begin{thm}
Let $X$ be a complete hyperbolic metric on $N$ and $\Be\in \hat{\pi}$ be a one-sided loop. For any $\Al\in \hat{\pi}$, if $P$ is an transverse intersection point of \( \Al(Y) \) and \( \Be(Y) \) and its forward angle $\theta$ satisfies the condition: 

\begin{align}
|\cos(\theta_{P})|>\tanh\left(\frac{(2m-1)l_{\Be}}{2}\right),
\end{align}

then $\Al$ has at least $m$ self-intersections. 

\end{thm}

\begin{cor}
Let $X$ be a complete hyperbolic metric on $N$ and $\Be\in \hat{\pi}$ be a one-sided loop. For any $\Al\in \hat{\pi}$, if $P$ is an ($\Al(X)$, $\Be(X)$)-intersection point and its forward angle $\theta$ satisfies the condition: 

\begin{align}
|\cos(\theta)|>\tanh\left(\frac{l_{\Be}}{2}\right),
\end{align}
then $\Al$ is not simple closed curve.
\end{cor}
\end{comment}

\end{document}